\documentclass[aap,preprint]{imsart}

\RequirePackage{amsthm,amsmath,amsfonts,amssymb}
\RequirePackage[numbers]{natbib}
\RequirePackage[colorlinks,citecolor=blue,urlcolor=blue]{hyperref}
\RequirePackage{graphicx}

\usepackage{todonotes}
\usepackage{bbm}
\usepackage{bm}

\startlocaldefs
\theoremstyle{plain}
\newtheorem{theorem}{Theorem}[section]
\newtheorem{lemma}[theorem]{Lemma}
\newtheorem{proposition}[theorem]{Proposition}

\newtheorem{remark}[theorem]{Remark}
\endlocaldefs

\newcommand{\E}{\mathbb{E}}
\newcommand{\N}{\mathbb{N}}
\renewcommand{\epsilon}{\varepsilon}

\begin{document}
\begin{frontmatter}
\title{Diffusion Limits at Small Times for Coalescent Processes with Mutation and Selection}
\runtitle{Diffusion Limits at Small Times for the ASG}
\begin{aug}
\author[A,B]{\fnms{Philip A.} \snm{Hanson}\ead[label=e1]{P.Hanson@warwick.ac.uk}},
\author[B,C,D]{\fnms{Paul A.} \snm{Jenkins}\ead[label=e2]{P.Jenkins@warwick.ac.uk}},
\author[B]{\fnms{Jere} \snm{Koskela}\ead[label=e3]{J.Koskela@warwick.ac.uk}}
\and
\author[B]{\fnms{Dario} \snm{Span\`o}\ead[label=e4]{D.Spano@warwick.ac.uk}}
\address[A]{Mathematics Institute, Zeeman Building,
University of Warwick, Coventry CV4 7AL, United Kingdom.}
\address[B]{Department of Statistics,
University of Warwick, Coventry CV4 7AL, United Kingdom.}
\address[C]{Department of Computer Science,
University of Warwick, Coventry CV4 7AL, United Kingdom.}
\address[D]{The Alan Turing Institute, British Library, London NW1 2DB, United Kingdom.\\
\printead{e1,e2,e3,e4}}%, \printead{e2}, \printead{e3}, \printead{e4}}
\end{aug}

\begin{abstract}
The Ancestral Selection Graph (ASG) is an important genealogical process which extends the well-known Kingman coalescent to incorporate natural selection. We show that the number of lineages of the ASG with and without mutation is asymptotic to $2/t$ as $t\to 0$, in agreement with the limiting behaviour of the Kingman coalescent. We couple these processes on the same probability space using a Poisson random measure construction that allows us to precisely compare their hitting times. These comparisons enable us to characterise the speed of coming down from infinity of the ASG as well as its fluctuations in a functional central limit theorem. This extends similar results for the Kingman coalescent.
\end{abstract}
\begin{keyword}[class=MSC2020]
\kwd[Primary ]{60J90}
\kwd{60F05}
\kwd[; secondary ]{60J80}
\end{keyword}

\begin{keyword}
\kwd{Coalescent}
\kwd{Mutation}
\kwd{Selection}
\kwd{Small Time Asymptotics}
\end{keyword}
\end{frontmatter}
\section{Introduction and Main Results}
\label{intro}
The Kingman coalescent \cite{KINGMAN1982235} is one of the fundamental models in the field of population genetics, though it has many mathematical qualities that are of independent interest. If one takes a large population of individuals and considers their ancestry into the past, a process that models the structure of this ancestry is the Kingman coalescent. Mathematically it is characterised as a continuous-time Markov process over the space of partitions of $\mathbb{N}=\{1,2,3,\dots\}$, where each pair of blocks in the partition coalesces independently at unit rate; each block representing an ancestral lineage.

This process can be initialised from the partition of singletons $\{\{1\},\{2\},\dots\}$, in which case the associated block-counting process $(N^{0,0}_t)_{t\geq0}$ starts at $+\infty$ and instantly becomes finite almost surely \cite{berestycki2009recent}. This phenomenon is called ``coming down from infinity'' and motivates investigation into the behaviour of $N^{0,0}_t$ close to $t=0$.

In \cite{Griffiths1984AsymptoticLD} Griffiths derives asymptotic expressions for $N^{0,0}_t$ when $t$ is small. There it is shown that, as $t\to0$, $N^{0,0}_t$ is approximately Gaussian with mean $2/t$ and variance $2/(3t)$. These approximations are often used when simulating the Wright-Fisher diffusion \cite{Jenkins_2017}.

When considering the ``coming down from infinity'' behaviour of the coalescent, a natural question to ask is how quickly this happens. Specifically, one looks for a function $\nu_t$ such that
\begin{equation}
    \lim_{t\to0}\frac{N^{0,0}_t}{\nu_t}=1,~~a.s.\label{Speed of coming down from infinity def}
\end{equation}
In \cite{bansaye2016speed} it is proved that for those general birth/death processes which come down from infinity, one can define $\nu$ as
\begin{equation}
    \nu_t:=\inf\{n\geq0;~\E_\infty[T_n]\leq t\},\label{Bansaye speed of cdi}
\end{equation}
where $T_n$ is the hitting time of $n$ and the subscript on $\E$ denotes that the process started from infinity. For the Kingman coalescent a known function satisfying \eqref{Speed of coming down from infinity def} is $2/t$ \cite[Theorem 4.9]{berestycki2009recent}. It should be noted that any function which satisfies $\nu_t\sim 2/t$ as $t\to0$ will also satisfy \eqref{Speed of coming down from infinity def}.

In \cite{limic2015} the authors fully characterised the fluctuations of $N^{0,0}_t$ as $t\to 0$ by proving the convergence of
\begin{equation}
    X^{0,0}_\epsilon(t):=\epsilon^{-1/2}\left(\frac{\epsilon t}{2}N^{0,0}_{\epsilon t}-1\right),~~t>0,~~~X^{0,0}_\epsilon(0)=0,\label{Kingman 2nd order pre limit}
\end{equation}
to
\begin{equation}
    Z_t:=\frac{1}{\sqrt{2}t}\int_0^tu\,{\rm{d}}W_u,~~t>0,~~~Z_0=0,\label{Gaussian Process Limit for ASG + Kingman}
\end{equation}
as $\epsilon\to 0$, where $W$ is a Brownian Motion. This convergence is in law in the Skorokhod space $D_\mathbb{R}([0,\infty))$ of real-valued c\`adl\`ag functions equipped with the $J_1$ topology that makes it a completely separable metric space.

Two key genetic mechanisms \emph{not} present in Kingman's original construction of the coalescent are mutation and selection. Parent-independent mutation is incorporated by having mutations occur at a constant rate $\theta/2 \geq 0$ independently on each line of ancestry, changing the genetic type of that line. This addition now means that the number of lineages moves from $n$ to $n-1$ at rate
\[\frac{n(n-1+\theta)}{2}.\]
Here we see that the addition of mutation only slightly affects the rate at which lineages are lost; asymptotically this expression is still $O(n^2)$ as $n\to\infty$. As such, one expects the same speed of coming down from infinity and indeed we prove this in Section \ref{Section3}.

The basic mechanic of selection is that one genetic type may have an evolutionary advantage over another. Forward in time we model this by having fit types reproduce at a higher rate. When looking backwards in time---without knowing the types of all individuals---there is uncertainty as to whether an offspring arose as a result of a normal reproduction event or one involving a fit type. In the ancestral process we model this by having lineages split into two independently at rate $\sigma/2\geq0$. Rather than a tree, the resulting genealogical structure is a graph known as the Ancestral Selection Graph (ASG), introduced to the coalescent framework by Neuhauser and Krone in \cite{Neuhauser519} and \cite{KRONE1997210}. These additions result in the number of lineages --- denoted $N_t^{\sigma,\theta}$ --- becoming a birth/death process with the following transition rates:
\begin{align*}
    n&\mapsto n-1\text{ at rate } \frac{n(n-1+\theta)}{2},\\
    n&\mapsto n+1\text{ at rate } \frac{\sigma n}{2}.
\end{align*}
In this paper we show that, close to $t=0$, the quadratic rate of coalescence dominates both the linear rate of upward jumps and the linear perturbation in the rate of downward jumps, and in fact the ASG has the same limiting behaviour as the coalescent in the following sense:
\begin{theorem}\label{maintheorem}
Let $(N_t^{\sigma,\theta})_{t\geq0}$ be the number of lineages in the ASG at time $t$. Then the process
\begin{equation}
X_\epsilon^{\sigma,\theta}(t):=\epsilon^{-1/2}\left(\frac{\epsilon t}{2}N_{\epsilon t}^{\sigma,\theta}-1\right),~~t>0,~~~X_\epsilon^{\sigma,\theta}(0)=0, \label{ASG 2nd order pre limit}
\end{equation}
converges in law in $D_\mathbb{R}([0,\infty))$ as $\epsilon\to0$ to the Gaussian process $Z$ given by \eqref{Gaussian Process Limit for ASG + Kingman}.
\end{theorem}
Note that if one sets $\sigma=0$ then the above result implies that the diffusion limit also holds for the Kingman coalescent with mutation. Furthermore, since the number of lineages in the Ancestral Recombination Graph (ARG) has a similar birth/death structure --- with a recombination rate $\rho\geq 0$ giving linear births --- the above theorem also applies to that model \cite{GriffithsMarjoramRecomb97}.

In the next section we outline a Poisson random measure construction that allows us to consider all of these processes together on the same probability space. Section \ref{Section3} contains an analysis of hitting times of these processes with Section \ref{Section4} containing the necessary lemmas that will be used to prove Theorem \ref{maintheorem} in Section \ref{Section5}.
\section{A Poisson Random Measure Construction}\label{Section2}
In order to analyse the number of lineages in the ASG and Kingman coalescent, a construction of them via a Poisson random measure (PRM) is very useful. Here we extend the construction of \cite{limic2015} to account for mutation and selection.

First, we define the following spaces:
\[\Delta:=\{(i,j):i,j\in\mathbb{N}, i<j\},\]
and
\[\bar{\Delta}:=\Delta\cup\{(i,\infty);i\in\mathbb{N}\}\cup\{(i,0);i\in\mathbb{N}\},\]
where a standard element of $\bar{\Delta}$ will be denoted $\boldsymbol{k}=(i,j)$. On a probability space $(\Omega,\mathcal{F}, \mathbb{P})$ define $\pi$ as a PRM on $\mathbb{R}_+\times\bar{\Delta}$ with intensity measure
\[
\nu:=\ell\otimes\left(\sum_{(i,j)\in\Delta}\delta_{(i,j)}+\sum_{i\in\mathbb{N}}\frac{\theta}{2}\delta_{(i,\infty)}+\sum_{i\in\mathbb{N}}\frac{\sigma}{2}\delta_{(i,0)}\right),
\]
where $\ell$ is the Lebesgue measure.

The coalescent can be obtained from $\pi$ as follows: an arrival of type ${(t,\boldsymbol{k})=(t,(i,j))}$ represents a potential coalescence, mutation or selection event. If $0<j<\infty$ then the lineages $i$ and $j$ coalesce if the process has at least $j$ lineages at time $t$. If $j=\infty$ then the $i$th lineage is lost --- again if the process has at least $i$ lineages at time $t$. If $j=0$ then the $i$th lineage is split into two new lineages.

We also let $\hat{\pi}:=\pi-\nu$ be the compensated PRM associated with $\pi$ and
\begin{align*}
  &\Delta_n:=\{(i,j)\in\Delta:1\leq i<j\leq n\},\\
  &\bar{\Delta}_n:=\Delta_n\cup\{(i,\infty)\in\bar{\Delta}:1\leq i\leq n\}\cup\{(i,0)\in\bar{\Delta}:1\leq i\leq n\}.
\end{align*}
If we consider starting the ASG with a finite number of lineages $n$, we can express the number of lineages at time $t$ in the following way:
\[N_{t,n}^{\sigma,\theta}:=n-\int_0^t\int_{\bar{\Delta}}\mathbbm{1}_{\bar{\Delta}_{N_{s-}^{\sigma,\theta}}}(\boldsymbol{k})\left[\mathbbm{1}_{j>0}(\boldsymbol{k})-\mathbbm{1}_{j=0}(\boldsymbol{k})\right]\pi({{\rm{d}}}s,{\rm{d}}\boldsymbol{k}).\]
We can then define the $\mathbb{P}$-almost sure limit
\[N_t^{\sigma,\theta} := \lim_{n\to\infty} N_{t,n}^{\sigma,\theta}\]
whose existence is assured by the lookdown construction of Donnelly and Kurtz \cite{donnelly1999}; see also Section \ref{Section5}, namely the proof of Lemma \ref{lemmaXdecomposition}.

The Kingman coalescent with and without mutation can be obtained from the above mechanics by thinning the different arrival types. With this we can construct all three processes --- $N_t^{0,0}$, $N_t^{0,\theta}$ and $N_t^{\sigma,\theta}$ --- on the same probability space $(\Omega,\mathcal{F},\mathbb{P})$ and obtain the inequalities
\begin{align}
  N_t^{0,\theta}(\omega)&\leq N_t^{0,0}(\omega),~~t\geq0,\label{PRM Ineq Kingman to Mutation}\\
  N_t^{0,\theta}(\omega)&\leq N_t^{\sigma,\theta}(\omega),~~t\geq0,\label{PRM Ineq Mutation to Selection}
  \end{align}
for almost every $\omega\in\Omega$. These inequalities follow immediately from the fact that any arrival of type $(i,j),$ $0<j<\infty$ that causes $N_t^{0,\theta}$ to decrease will also cause $N_t^{0,0}$ and $N_t^{\sigma,\theta}$ to decrease. Furthermore, any mutation arrival that causes $N_t^{0,\theta}$ to decrease will also cause $N_t^{\sigma,\theta}$ to decrease, with the latter process occasionally moving upwards, hence \eqref{PRM Ineq Mutation to Selection}.

From here on, when used with any three of these processes, $\E$ is the expectation operator of the probability space outlined in this section.
\section{An Analysis of Hitting Times}\label{Section3}
In the study of general birth/death processes one can learn a lot about the behaviour of a process through analysis of its hitting times. Of course in the case of the Kingman coalescent with mutation, the time taken to get from $n$ to $n-1$ is simply an exponential clock:
\[T_{n,n-1}^{0,\theta}\sim\text{Exp}\left(\frac{n(n-1+\theta)}{2}\right),\]
which has a finite $k$th moment for all $k\in\mathbb{N}$. The ASG however experiences both births and deaths and so the situation is more complex. A first-step analysis leads to the following recurrence relation:
\begin{equation}
    T_{n,n-1}^{\sigma,\theta}\overset{d}{=}\xi_n+\mathbbm{1}_{E_n}\left(\widehat{T}_{n+1,n}^{\sigma,\theta}+\widehat{T}_{n,n-1}^{\sigma,\theta}\right),\label{ASG hitting time recursion}
  \end{equation}
 where $\xi_n\sim\text{Exp}\left(n(n-1+\theta+\sigma)/2\right)$ is the holding time at level $n$, $E_n$ is the event that the ASG jumps up after reaching level $n$ for the first time and $\widehat{T}_{n+1,n}^{\sigma,\theta}$ and $\widehat{T}_{n,n-1}^{\sigma,\theta}$ are independent copies of $T_{n+1,n}^{\sigma,\theta}$ and $T_{n,n-1}^{\sigma,\theta}$ respectively. This recurrence relation will be used to prove the asymptotic behaviour of the moments of $T_{n,n-1}^{\sigma,\theta}$, but first we state a lemma on conditions needed for a general birth/death process to have finite $k$th moments for all $k\in\mathbb{N}$:
 \begin{lemma}\label{lemma3.1}
 Let $T_{n,n-1}$ be the hitting time of $n-1$ from $n$ of a birth/death process with birth and death rates $\lambda_n$ and $\mu_n$ respectively. Also suppose
 \begin{equation}
     \limsup_{n\to\infty}\frac{\lambda_n}{\mu_n}<1.\label{Limsup Birth/Death rates condition}
 \end{equation}
 Then there exists $n_0\in\mathbb{N}$ such that
 \begin{equation}
     \E\left[\sup_{n\geq n_0}T_{n,n-1}^k\right]<\infty\label{Bounded hitting time moments, b/d process}
 \end{equation}
 for all $k\in\mathbb{N}$.
 \end{lemma}
\begin{remark}
 It should be noted that tighter asymptotic expressions for the above moments can be found in \cite[Lemma 2.2]{bansaye2016speed} in a less general setting. There the authors consider only the first three moments, and more restrictions are placed on the birth and death rates.
\end{remark}
\begin{proof}
  Firstly, note that condition \eqref{Limsup Birth/Death rates condition} allows us to pick an $n_0$ such that
  \begin{equation}
      \sup_{n\geq n_0}\frac{\lambda_n}{\mu_n+\lambda_n}=:p<\frac{1}{2},\label{p def}
  \end{equation}
  Now, let $\left(\Xi_t\right)_{t\geq0}$ be a simple random walk with up/down transition probabilities of $p$, $1-p$ respectively with holding times equal to that of the birth/death process, and let its hitting time of $n-1$ from $n$ be denoted by $T$. Then, for $n\geq n_0$, $T_{n,n-1}$ is stochastically dominated by $T$.

  If we consider what happens after one step of $\Xi$ we obtain the equality
  \begin{equation}
      T\overset{d}{=}\zeta_n+\mathbbm{1}_{\hat{E}_n}(\hat{T}+\widetilde{T}),\label{Xi hitting time recursion}
  \end{equation}
  where $\zeta_n\sim\text{Exp}(\lambda_n+\mu_n)$ and $\widehat{E}_n$ is the event that $\Xi$ jumped up after this time, with $\widehat{T}$ and $\widetilde{T}$ being independent copies of $T$. Taking expectation of \eqref{Xi hitting time recursion} we see that
  \[\E[T]=\E[\zeta_n]+2p\E[T]\]
  and so since $p<1/2$, $\E[T]<\infty$. Now, let $k\in\mathbb{N}$ and suppose $\E\left[T^j\right]<\infty$ for all $j< k$. If we take the $k$th moment of \eqref{Xi hitting time recursion} we see that
  \begin{align*}
      \E\left[T^k\right]&=\sum_{j=0}^{k}\binom{k}{j}p\E\left[\zeta_n^{k-j}\right]\E\left[(\widehat{T}+\widetilde{T})^j\right],\\
      &=\sum_{j=0}^{k}\binom{k}{j}p\E\left[\zeta_n^{k-j}\right]\sum_{i=0}^j\binom{j}{i}\E\left[T^i\right]\E\left[T^{j-i}\right].
  \end{align*}
  Now on the right hand side the coefficient of $\E\left[T^k\right]$ is $2p$, coming from the $j=k$, $i=0,k$ terms. Rearranging this and using the fact that $p<1/2$ we have finiteness of $\E\left[T^k\right]$. We can then inductively obtain finiteness of all moments of $T$. This immediately bounds all moments of $T_{n,n-1}$ for $n\geq n_0$ by the corresponding moment for $T$ and so \eqref{Bounded hitting time moments, b/d process} follows.
\end{proof}
With this established we proceed by considering the moments of these hitting times for the ASG. A lower bound, $\E[(T_{n,n-1}^{0,\theta})^k]\leq\E[(T_{n,n-1}^{\sigma,\theta})^k]$, is immediate. The following proposition provides a corresponding upper bound:
\begin{proposition}
\label{Tn,n-1 comparison Prop}
For sufficiently large $n$,
  \begin{equation}
  \E\left[\left(T_{n,n-1}^{\sigma,\theta}\right)^k\right]\leq\E\left[\left(T_{n,n-1}^{0,\theta}\right)^k\right]+\frac{B_k^{\sigma,\theta}}{n^{2k+1}}.\label{ASG Tn,n-1 kth moment ineqs}
  \end{equation}
  where $B_k^{\sigma,\theta}\in\mathbb{R}^+$.
\end{proposition}
\begin{proof}
  We start our proof by noting that
  \begin{equation}
    \E \left[\xi_n^k\right]\leq\E\left[\left(T_{n,n-1}^{0,\theta}\right)^k\right].\label{xi n Ineq}
  \end{equation}
This substitution will be used since a comparison of the ASG to the Kingman coalescent is the goal. Now, we take both sides of \eqref{ASG hitting time recursion} to the power $k$ and apply expectations, remembering the independence of the event $E_n$. Using the inequalities \eqref{xi n Ineq} and $(x+y)^j\leq2^j(x^j+y^j)$ for $j \in \mathbb{N}$ and $x,y\geq 0$ we get 
  \begin{align*}
  \E\left[\left(T_{n,n-1}^{\sigma,\theta}\right)^k\right]\leq {}& \E\left[\left(T_{n,n-1}^{0,\theta}\right)^k\right]\\
  & {}+\sum_{j=1}^{k}\binom{k}{j}\frac{2^j\sigma}{n-1+\theta+\sigma}\E\left[\xi_n^{k-j}\left(\left(\widehat{T}_{n+1,n}^{\sigma,\theta}\right)^j+\left(\widehat{T}_{n,n-1}^{\sigma,\theta}\right)^j\right)\right].
  \end{align*}
  Applying the independence of the $\xi_n$ random variable to each term in the sum, using \eqref{xi n Ineq} again and letting $a_{n,k}:=\E[(T_{n,n-1}^{\sigma,\theta})^k]$, $x_{n,k}:=\E[(T_{n,n-1}^{0,\theta})^k]$, we obtain the following recursion formula:
  \begin{equation}
  a_{n,k}\leq x_{n,k}+\sum_{j=1}^{k}\binom{k}{j}\frac{2^j\sigma}{n-1+\theta+\sigma}x_{n,k-j}\left(a_{n+1,j}+a_{n,j}\right).\label{a_n,k b_n,k Recursion}
  \end{equation}
  Next we need to show that $a_{n,k}=O(n^{-2k})$ as $n\to\infty$. We start by noting that the ASG satisfies \eqref{Limsup Birth/Death rates condition} and so \eqref{Bounded hitting time moments, b/d process} holds meaning $a_{n,k}=O(1)$ as $n\to\infty$. We can now recursively use \eqref{a_n,k b_n,k Recursion} to improve these asymptotics. To this end, suppose that $a_{n,k}=O(n^{-i})$ for some $0\leq i\leq 2k-1$ so that there exists a constant $C_i$ such that $a_{n,k} \leq C_in^{-i}$ for all $n$ greater than or equal to some $n_0$. Since $x_{n,j}=j!2^{j}n^{-j}(n-1+\theta)^{-j}$ there exists $B_{j}^{0,\theta}$ such that
  \begin{equation}
      x_{n,j}\leq\frac{B_{j}^{0,\theta}}{n^{2j}},\qquad \forall j\in\N.\label{Kingman Tn,n-1 bound}
  \end{equation}
  Substituting these upper bounds into \eqref{a_n,k b_n,k Recursion} we obtain the following:
  \[a_{n,k}\leq\frac{B_k^{0,\theta}}{n^{2k}}+\sum_{j=1}^{k}\binom{k}{j}\frac{2^j\sigma}{n-1+\theta+\sigma}\left(\frac{B_{k-j}^{0,\theta}}{n^{2(k-j)}}\right)\left(\frac{C_i}{n^i}+\frac{C_i}{n^i}\right).\]
  Multiplying this inequality by $n^{i+1}$ we get
  \[n^{i+1}a_{n,k}\leq\frac{B_k^{0,\theta}n^{i+1}}{n^{2k}}+\sum_{j=1}^{k}\binom{k}{j}\frac{2^{j+1}B_k^{0,\theta}C_i\sigma}{n-1+\theta+\sigma}\left(n^{-2k+2j}\times n^{-i}\times n^{i+1}\right).\]
  Tallying up the indices of $n$ in the bracket we get $-2k+2j+1$. This expression is maximised at $j=k$, where it is equal to $1$. Thanks to the preceding term outside the bracket we see that summand as a whole is $O(1)$ as $n\to\infty$. This means that in fact $a_{n,k}=O(n^{-i-1})$ and so we recursively apply this argument until we obtain $a_{n,k}=O(n^{-2k})$.

  To conclude the proof we multiply \eqref{a_n,k b_n,k Recursion} by $n^{2k}$:
  \begin{align*}
      n^{2k}a_{n,k}\leq{}& n^{2k}x_{n,k}\\
      &{}+\sum_{j=1}^k\binom{k}{j}\frac{2^j\sigma}{n-1+\theta+\sigma}n^{2k}x_{n,k-j}\left(a_{n+1,j}+a_{n,j}\right)
  \end{align*}
  and note that the terms in the sum are all $O(n^{-1})$ as $n\to\infty$. Thus we have
  \[n^{2k}a_{n,k}\leq n^{2k}x_{n,k}+O(n^{-1}),\]
  yielding \eqref{ASG Tn,n-1 kth moment ineqs}.
\end{proof}
We now consider the hitting times $T_n^{\sigma,\theta}:=T_{\infty,n}^{\sigma,\theta}$. Before looking at this expression for the ASG we state a brief lemma on the large $n$ asymptotics of the hitting time $T_n^{0,\theta}$ for the Kingman coalescent with mutation:
\begin{lemma}
\label{Kingman Hitting Time Asymptotics Lemma}
  For all $k\in\mathbb{N}$ there exists $C_k^{0,\theta}\in\mathbb{R}^+$ such that
  \begin{equation}
    \E\left[\left(T_n^{0,\theta}\right)^k\right]\leq\frac{C_k^{0,\theta}}{n^k} \qquad \forall n\in\mathbb{N}.\label{Kingman T_n Asymptotics}
  \end{equation}
\end{lemma}
\begin{proof}
  For this proof we will take advantage of the fact that the hitting time $T_n^{0,\theta}$ can be written as a sum of the hitting times $T_{i,i-1}^{0,\theta}$:
  \[T_n^{0,\theta}=\sum_{i\geq n+1}T_{i,i-1}^{0,\theta}.\]
  Taking the $k$th power of this expression we get
    \begin{equation*}
    \left(T_n^{0,\theta}\right)^k\overset{d}{=}\left(\sum_{i=n+1}^\infty T_{i,i-1}^{0,\theta}\right)^k\overset{d}{=}\sum_{j=1}^k\sum_{n+1\leq i_1<\dots<i_j}\sum_{\substack{\boldsymbol{m}\in\mathbb{N}^j\\|\boldsymbol{m}|=k}}\binom{k}{\boldsymbol{m}}\prod_{l=1}^j\left(T_{i_l,i_l-1}^{0,\theta}\right)^{m_l}.
  \end{equation*}
  Since the $i_l$ are distinct we can apply expectations to the above and separate out each of the terms in the product by independence. Doing so yields the following:
  \begin{align*}
    \E\left[\left(T_n^{0,\theta}\right)^k\right]&=\sum_{j=1}^k\sum_{n+1\leq i_1<\dots<i_j}\sum_{\substack{\boldsymbol{m}\in\mathbb{N}^j\\|\boldsymbol{m}|=k}}\binom{k}{\boldsymbol{m}}\prod_{l=1}^j\E\left[\left(T_{i_l,i_l-1}^{0,\theta}\right)^{m_l}\right]\\
    &\leq\sum_{j=1}^k\sum_{n+1\leq i_1<\dots<i_j}\sum_{\substack{\boldsymbol{m}\in\mathbb{N}^j\\|\boldsymbol{m}|=k}}\binom{k}{\boldsymbol{m}}\prod_{l=1}^j\frac{B_{m_l}^{0,\theta}}{i_l^{2m_l}}\\
    &\leq\sum_{j=1}^k\sum_{\substack{\boldsymbol{m}\in\mathbb{N}^j\\|\boldsymbol{m}|=k}}\binom{k}{\boldsymbol{m}}\prod_{l=1}^j\sum_{i_l\geq n+1}\frac{B_{m_l}^{0,\theta}}{i_l^{2m_l}}\\
    &\leq\sum_{j=1}^k\sum_{\substack{\boldsymbol{m}\in\mathbb{N}^j\\|\boldsymbol{m}|=k}}\binom{k}{\boldsymbol{m}}\prod_{l=1}^j\int_n^\infty\frac{B_{m_l}^{0,\theta}}{x^{2m_l}}\,{\rm{d}}x\\
    &=\sum_{j=1}^k\sum_{\substack{\boldsymbol{m}\in\mathbb{N}^j\\|\boldsymbol{m}|=k}}\binom{k}{\boldsymbol{m}}\prod_{l=1}^j\frac{B_{m_l}^{0,\theta}}{(2m_l-1)n^{2m_l-1}}\\
    &=\sum_{j=1}^k\sum_{\substack{\boldsymbol{m}\in\mathbb{N}^j\\|\boldsymbol{m}|=k}}\binom{k}{\boldsymbol{m}}\frac{1}{n^{2k-j}}\prod_{l=1}^j\frac{B_{m_l}^{0,\theta}}{(2m_l-1)}\\
    &\leq\frac{C_k^{0,\theta}}{n^k},
  \end{align*}
  for some $C_k^{0,\theta}\in\mathbb{R}^+$,where $|\boldsymbol{m}|$ is simply the $L^1$ norm.
\end{proof}
Using Proposition \ref{Tn,n-1 comparison Prop} and a similar approach as the previous lemma we can analyse the large $n$ asymptotics of $T_n^{\sigma,\theta}$. Again we find that the hitting times for the coalescent and the ASG are ``close'' in the sense that the immediate bound $\E[(T_n^{0,\theta})^k]\leq\E[(T_n^{\sigma,\theta})^k]$ is complemented by the following result:
\begin{proposition}\label{prop3.5} For $n$ sufficiently large, we have
  \begin{equation}
    \E\left[\left(T_n^{\sigma,\theta}\right)^k\right]\leq\E\left[\left(T_n^{0,\theta}\right)^k\right]+\frac{C_k^{\sigma,\theta}}{n^{k+1}},\label{ASG T_n Kingman comparison Ineq}
  \end{equation}
  where $C_k^{\sigma,\theta}\in\mathbb{R}^+.$
\end{proposition}
\begin{proof}
  Similar to the previous lemma, we can write the $k$th moment of $T_n^{\sigma,\theta}$ in the following way:
  \begin{equation}
    \E\left[\left(T_n^{\sigma,\theta}\right)^k\right]=\sum_{j=1}^k\sum_{n+1\leq i_1<\dots<i_j}\sum_{\substack{\boldsymbol{m}\in\mathbb{N}^j\\|\boldsymbol{m}|=k}}\binom{k}{\boldsymbol{m}}\prod_{l=1}^j\E\left[\left(T_{i_l,i_l-1}^{\sigma,\theta}\right)^{m_l}\right].\label{E[T_n^k] formula}
  \end{equation}
  Substituting the inequality from \eqref{ASG Tn,n-1 kth moment ineqs} into \eqref{E[T_n^k] formula} we obtain the following:  \begin{equation}
    \E\left[\left(T_n^{\sigma,\theta}\right)^k\right]\leq\sum_{j=1}^k\sum_{n+1\leq i_1<\dots<i_j}\sum_{\substack{\boldsymbol{m}\in\mathbb{N}^j\\|\boldsymbol{m}|=k}}\binom{k}{\boldsymbol{m}}\prod_{l=1}^j\left(\E\left[\left(T_{i_l,i_l-1}^{0,\theta}\right)^{m_l}\right]+\frac{B_{m_l}^{\sigma,\theta}}{i_l^{2m_l+1}}\right).\label{Big T_n^k ineq sum+product}
  \end{equation}
  From this point on we will use the shorthand $x_{i,m}=\E\left[\left(T_{i,i-1}^{0,\theta}\right)^{m}\right]$ and $y_{i,m}:=B_{m}^{\sigma,\theta}/i^{2m+1}$. We also let $B_{m}^{0,\theta}$ be the constant from \eqref{Kingman Tn,n-1 bound} such that $x_{i,m}\leq B_{m}^{0,\theta}/i^{2m}$ and consider $n$ large enough so that $x_{i,m}>y_{i,m}$ for all $i\geq n$ and $1\leq k\leq m$.

  Thinking about expanding the brackets inside the product we see that the leading term recovers exactly $\E\left[\left(T_n^{0,\theta}\right)^k\right]$. Meanwhile, we can encode the remaining terms as follows:
  \begin{equation}
      \E\left[\left(T_n^{\sigma,\theta}\right)^k\right]\leq\E\left[\left(T_n^{0,\theta}\right)^k\right]+\sum_{j=1}^k\sum_{n+1\leq i_1<\dots<i_j}\sum_{\substack{\boldsymbol{m}\in\mathbb{N}^j\\|\boldsymbol{m}|=k}}\binom{k}{\boldsymbol{m}}\sum_{\substack{{\boldsymbol \alpha}\in\{0,1\}^j\\|{\boldsymbol \alpha}|\neq j}}\prod_{l=1}^jx_{i_l,m_l}^{\alpha_l}y_{i_l,m_l}^{1-\alpha_l}.\label{alpha sum formula}
  \end{equation}
  Now, since we are considering $n$ large enough so that $x_{i_l,m_l}>y_{i_l,m_l}$, the largest term in the sum is the one of the form $x_{i_1,m_1} \dots x_{i_{l-1},m_{l-1}} y_{i_l,m_l} x_{i_{l+1},m_{l+1}} \dots x_{i_j,m_j}$ for some $l=1,\dots,j$. We do not know which one of these terms will be largest so we use the following upper bound:
  \begin{align*}
    \E\left[\left(T_n^{\sigma,\theta}\right)^k\right]\leq{}&\E\left[\left(T_n^{0,\theta}\right)^k\right]\\
    &{}+\sum_{j=1}^k\sum_{n+1\leq i_1<\dots<i_j}\sum_{\substack{\boldsymbol{m}\in\mathbb{N}^j\\|\boldsymbol{m}|=k}}\binom{k}{\boldsymbol{m}}(2^j-1)\sum_{l=1}^jy_{i_l,m_l}\prod_{p\neq l}x_{i_p,m_p}.
  \end{align*}
  Note that the factor of $2^j-1$ appears as this is the total number of terms in the sum over ${\boldsymbol \alpha}$ in \eqref{alpha sum formula}. We can now bound this above by separating out the sums over the different indices, dropping the restriction of $i_1<\ldots<i_j$, and using the bound $x_{i_l,m_l}\leq B_{m_l}^{0,\theta}/i_l^{2m_l}$. This gives us
  \begin{align*}
    \E\left[\left(T_n^{\sigma,\theta}\right)^k\right]\leq {}& \E\left[\left(T_n^{0,\theta}\right)^k\right]\\
    & {}+\sum_{j=1}^k\sum_{\substack{\boldsymbol{m}\in\mathbb{N}^j\\|\boldsymbol{m}|=k}}\binom{k}{\boldsymbol{m}}(2^j-1)\sum_{l=1}^j\sum_{i_l\geq n+1}\frac{B_{m_l}^{\sigma,\theta}}{i_l^{2m_l+1}}\prod_{p\neq l}\sum_{i_p\geq n+1}\frac{B_{m_p}^{0,\theta}}{i_p^{2m_p}}.
  \end{align*}
  Using integral bounds on each of these sums we obtain the following:
  \begin{align*}
    \E&\left[\left(T_n^{\sigma,\theta}\right)^k\right]\leq\E\left[\left(T_n^{0,\theta}\right)^k\right]\\
    &+\sum_{j=1}^k\sum_{\substack{\boldsymbol{m}\in\mathbb{N}^j\\|\boldsymbol{m}|=k}}\binom{k}{\boldsymbol{m}}(2^j-1)\sum_{l=1}^j\int_n^\infty\frac{B_{m_1}^{0,\theta}}{x^{2m_1}}\,{\rm{d}}x\dots\int_n^\infty\frac{B_{m_l}^{\sigma,\theta}}{x^{2m_l+1}}\,{\rm{d}}x\dots\int_n^\infty\frac{B_{m_{j}}^{0,\theta}}{x^{2m_{j}}}\,{\rm{d}}x\\
    ={}&\E\left[\left(T_n^{0,\theta}\right)^k\right]+\sum_{j=1}^k\sum_{\substack{\boldsymbol{m}\in\mathbb{N}^j\\|\boldsymbol{m}|=k}}\binom{k}{\boldsymbol{m}}(2^j-1)\sum_{l=1}^j\frac{B_{m_l}^{\sigma,\theta}}{2m_ln^{2m_l}}\prod_{p\neq l}\frac{B_{m_p}^{0,\theta}}{(2m_p-1)n^{2m_p-1}}.
  \end{align*}
  Considering that since all sums are now finite and the $m_i$ sum to $k$, the right hand side becomes
  \begin{align*}
    \E\left[\left(T_n^{\sigma,\theta}\right)^k\right]\leq{}&\E\left[\left(T_n^{0,\theta}\right)^k\right]\\
    &{}+\sum_{j=1}^k\sum_{\substack{\boldsymbol{m}\in\mathbb{N}^j\\|\boldsymbol{m}|=k}}\binom{k}{\boldsymbol{m}}(2^j-1)\sum_{l=1}^j\left(\frac{B_{m_j}^{\sigma,\theta}}{2m_j}\prod_{i\neq l}\frac{B_{m_i}^{0,\theta}}{(2m_i-1)}\right)\frac{1}{n^{2k-j+1}}.
  \end{align*}
  Since all the sums above are finite then there exist $D_j^{\sigma,\theta}\in\mathbb{R}^+$, $j=1,\dots,k$ such that
  \begin{equation*}
    \E\left[\left(T_n^{\sigma,\theta}\right)^k\right]\leq\E\left[\left(T_n^{0,\theta}\right)^k\right]+\sum_{j=1}^k\frac{D_j^{\sigma,\theta}}{n^{2k-j+1}},
  \end{equation*}
  and so we have \eqref{ASG T_n Kingman comparison Ineq}.
\end{proof}
\section{Controlling the ASG at Small Times}\label{Section4}
In this section we discuss the speed of coming down from infinity for the ASG and establish the mode of its convergence to that speed.

Before talking about the ASG we formally establish that the Kingman coalescent with mutation has the same speed of coming down from infinity as that without mutation. It should be noted that we consider this an unsurprising result but have not seen a formal proof, and so we present one here:
\begin{proposition}\label{prop4.1}
  \begin{equation}
    \lim_{t\to0}\frac{tN_t^{0,\theta}}{2}=1~~~\text{a.s}.\label{Kingman with mutation speed of CDI}
  \end{equation}
  \end{proposition}
\begin{proof}
  In order to ensure that there exists a function $\nu_t$ such that $N_t^{0, \theta} / \nu_t$ converges to 1 almost surely, we turn to \cite[Theorem 5.1]{bansaye2016speed} where this limit is considered for general birth/death processes with birth and death rates $\lambda_n$ and $\mu_n$ respectively. There, two conditions are sufficient to ensure the existence of $\nu$:
  \begin{itemize}
    \item $\lim_{n\to\infty}\lambda_n/\mu_n=0$
    \item $(\mu_n)_{n\geq1}$ varies regularly with index $\rho>1$.
  \end{itemize}
   In our case $\lambda_n=0$ so there is only the second condition to check: that the death rate $\mu_n$ varies regularly with an index $\rho>1$ as $n\to\infty$. Recall that a sequence of real nonzero numbers $(a_n)_{n\geq1}$ varies regularly with index $\rho\neq0$ if, for all $b>0$,
   \[\lim_{n\to\infty}\frac{a_{[bn]}}{a_n}=b^\rho.\]
  Since our death rates are a quadratic polynomial they satisfy the above condition with $\rho=2$. This gives us our a.s. convergence.

  We proceed to verify the form of $\nu$ by recalling \eqref{Bansaye speed of cdi} and considering the expected hitting time of $n$ by $N^{0,\theta}$:
  \[\E\left[T^{0,\theta}_n\right]=\sum_{k=n+1}^\infty\frac{2}{k(k-1+\theta)}.\]
  The asymptotic expression for which is found by an integral approximation to sums:
  \begin{equation}
    \int_{n+1}^\infty\frac{2}{x(x+\theta -1)}\,{{\rm{d}}}x\leq\E \left[T^{0,\theta}_n\right]\leq\int_n^\infty\frac{2}{x(x+\theta-1)}\,{{\rm{d}}}x.\nonumber
  \end{equation}
  Making a substitution $x=ny$ we have
  \begin{align*}
    \int_n^\infty\frac{2}{x(x-1+\theta)}\,{{\rm{d}}}x=\frac{2}{n}\int_1^\infty\frac{1}{y(y+\frac{\theta-1}{n})}\,{{\rm{d}}}y
    &=\frac{2}{\theta-1}\ln\left(1+\frac{\theta-1}{n}\right)
    =\frac{2}{n}+O(n^{-2}),
  \end{align*}
and hence $\E[T^{0,\theta}_n] = 2/n + O(n^{-2})$. To see that this implies that $\nu_t$ is asymptotically equivalent to $2/t$ as $t\to0$ we use the definition of $\nu_t$ to get the following set of inequalities:
  \begin{align}
    \E\left[T^{0,\theta}_{\nu_t}\right]\leq t&<\E\left[T^{0,\theta}_{\nu_t-1}\right]\label{Speed of coming down with hitting time ineqs}\\
    \frac{2}{\nu_t}+O((\nu_t)^{-2})\leq t&<\frac{2}{\nu_t-1}+O((\nu_t)^{-2})\nonumber\\
    \frac{2}{t}+O((\nu_tt)^{-1})\leq \nu_t&<\frac{2}{t}\left(\frac{\nu_t}{\nu_t-1}\right)+O((\nu_tt)^{-1}).\nonumber
  \end{align}
  Now, when we divide the error term on either side by $2/t$ we get a function that grows at most like $(\nu_t)^{-1}$ and so goes to zero as $t\to0$. This tells us that $\nu_t$ is asymptotically equivalent to $2/t$ as $t\to0$.
\end{proof}
Thanks to the asymptotics established in the previous section, the equivalent result for the ASG is now simple to prove:
\begin{proposition}\label{prop4.2}
  \begin{equation}
    \lim_{t\to0}\frac{tN_t^{\sigma,\theta}}{2}=1~~~\text{a.s}.\label{ASG Speed of CDI}
  \end{equation}
\end{proposition}
\begin{proof}
  Again we appeal to \cite[Theorem 5.1]{bansaye2016speed} in order to verify the above limit. For the ASG, the death rates are the same as Kingman with mutation so $\mu_n$ still varies regularly with index 2. Furthermore, since our birth rates are linear in $n$ we still have that $\lambda_n/\mu_n\to0$ as $n\to\infty$. Thus \cite[Theorem 5.1]{bansaye2016speed} gives us our almost sure convergence.

  Looking at the equivalent of equation \eqref{Speed of coming down with hitting time ineqs} for the ASG we can use \eqref{ASG T_n Kingman comparison Ineq} to state the following series of inequalities for sufficiently small $t$:
  \begin{align*}
    \E\left[T^{0,\theta}_{\nu_t}\right]\leq\E\left[T^{\sigma,\theta}_{\nu_t}\right]\leq t&<\E\left[T^{\sigma,\theta}_{\nu_t-1}\right]\leq \E\left[T_{\nu_t-1}^{0,\theta}\right]+\frac{C_1^{\sigma,\theta}}{(\nu_t-1)^2}\\
    \frac{2}{\nu_t}+O((\nu_t)^{-2})\leq t&<\frac{2}{\nu_t-1}+O((\nu_t)^{-2})\\
    \frac{2}{t}+O((\nu_tt)^{-1})\leq \nu_t&<\frac{2}{t}\left(\frac{\nu_t}{\nu_t-1}\right)+O((\nu_tt)^{-1}).
  \end{align*}
  In the same way as in the previous proof we have that $\nu_t\sim 2/t$ as $t\to0$ and so \eqref{ASG Speed of CDI} is verified.
\end{proof}
In order to establish Theorem \ref{maintheorem} the above almost sure convergence is not enough. In fact, in order to verify the convergence of \eqref{ASG 2nd order pre limit} to \eqref{Gaussian Process Limit for ASG + Kingman} we need more control over the process near zero. Namely we need a result analogous to \cite[Theorem 2]{berestycki2010} for the ASG. We first establish separately the result under neutrality as it will be used in the proof of the equivalent statement when selection is present:
\begin{proposition}\label{Kingman sup lim prop}
  \begin{equation}
    \lim_{t\to0}\E\left[\sup_{s\leq t}\left(\frac{sN_s^{0,\theta}}{2}-1\right)^k\right]=0,\qquad\forall k\in\mathbb{N}.\label{Kingman with Mutation E sup limit k}
  \end{equation}
\end{proposition}
\begin{proof}
  Using \cite[Theorem 2]{berestycki2010} with $d=k$ on $N_s^{0,0}$, there exists a time $t_1>0$ such that
  \[\E\left[\sup_{s\leq t_1}\left(\frac{sN_s^{0,0}}{2}-1\right)^k\right]\leq 1.\]
  This combined with the fact that $(sN_s^{0,0}/2)^k$ can grow at most like $s^k$ for $s>t_1$ ensures that we have
  \[\E\left[\sup_{s\leq t}\left(\frac{s}{2}N_s^{0,0}\right)^k\right]<\infty,\]
  for all $t\geq0$, $k\in\mathbb{N}$. This now immediately gives us
  \begin{equation}
    \E\left[\sup_{s\leq t}\left(\frac{s}{2}N_s^{0,\theta}\right)^k\right]<\infty,\label{Kingman sup bound k}
  \end{equation}
  for all $k\in\mathbb{N}$ thanks to \eqref{PRM Ineq Kingman to Mutation}.
 In order to establish \eqref{Kingman with Mutation E sup limit k} we first show the almost sure convergence of
 \begin{equation}
     \sup_{s\leq t}\left(\frac{sN_s^{0,\theta}}{2}-1\right)^k\label{sup ASG process k}
 \end{equation}
 to zero as $t\to 0$. This is easily done with an application of the continuous mapping theorem and \eqref{Kingman with mutation speed of CDI}; in fact since we are considering a right limit we rely only on the right continuity of the function and the supremum from the right. With this, we then bound \eqref{sup ASG process k} above by $\sup_{s\leq t}(sN_s^{0,\theta}/2)^k\vee 1$ and use the dominated convergence theorem together with \eqref{Kingman sup bound k}
 to obtain \eqref{Kingman with Mutation E sup limit k}.
\end{proof}
In order to obtain the same result for the ASG we need to take a closer look at the behaviour of the process $(sN_s^{\sigma,\theta}/2)^k$. We consider what happens to this process between successive hitting times $T_n^{\sigma,\theta}$ and $T_{n-1}^{\sigma,\theta}$. Between these times the process will sit on one of the curves in Figure \ref{figure}.
\begin{figure}[ht]
  \centerline{\includegraphics[scale=.25]{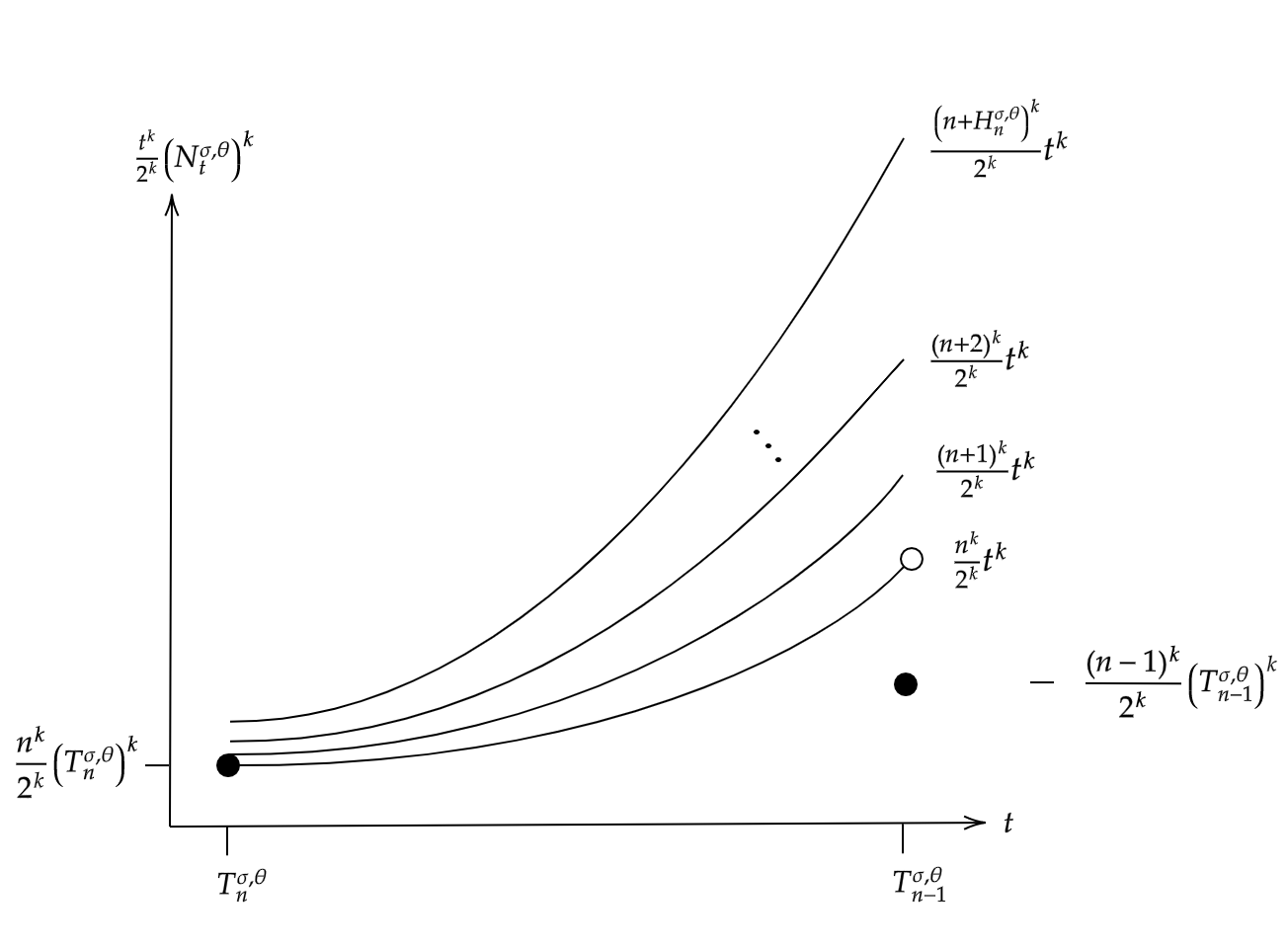}}
  \caption{\label{figure}Possible trajectories of the function $(sN_s^{\sigma,\theta}/2)^k$ between the first hitting times of level $n$ and $n-1$.}
\end{figure}
Here the random variable $H_n^{\sigma,\theta}$ is defined as
\[H_n^{\sigma,\theta}:=\#\{s\in[T_n^{\sigma,\theta},T_{n-1}^{\sigma,\theta}):N_s^{\sigma,\theta}-N_{s-}^{\sigma,\theta}>0\}\]
 or simply the number of times the ASG jumps up in this time window. From the diagram we see that the highest point that the process can reach is
\begin{equation}
  \frac{\left(n+H_n^{\sigma,\theta}\right)^k}{2^k}\left(T_{n-1}^{\sigma,\theta}\right)^k=\sum_{j=0}^k\binom{k}{j}\frac{n^{k-j}}{2^k}\left(H_n^{\sigma,\theta}\right)^j\left(T_{n-1}^{\sigma,\theta}\right)^k.\label{Bound with Hn,n and Tn}
\end{equation}
We can now proceed by looking at the supremum of the above random variables over all $n$ since
\begin{equation}
  \sup_{s\leq t}\left(\frac{sN_s^{\sigma,\theta}}{2}\right)^k\leq \sup_{n\in\mathbb{N}}\frac{\left(n+H_n^{\sigma,\theta}\right)^k}{2^k}\left(T_{n-1}^{\sigma,\theta}\right)^k,\label{From process to hitting times ineq}
\end{equation}
for any $t\geq 0$. Before stating our result however, we need to know more about the moments of $H_n^{\sigma,\theta}$. In the proof of \cite[Lemma 4.1]{bansaye2016speed} it is shown, for a general birth/death process $N_t$ with birth/death rates $\lambda_n$ and $\mu_n$, that if \eqref{Limsup Birth/Death rates condition} holds then the associated $H_n$ random variable
\[H_n:=\#\{s\in[T_n,T_{n-1}):N_s-N_{s-}>0\}\]
satisfies
\[\E_\infty\left[H_n^2\right]\leq C\frac{\lambda_n}{\mu_n},\]
for some constant $C\geq0$, for all $n\in\N$. Since we need to consider higher moments of $H_n^{\sigma,\theta}$, this inequality is not enough. As such, we extend the above inequality to all moments of $H_n$ for general birth/death processes:
\begin{lemma}\label{lemma4.4}
Let $(N_t)_{t\geq0}$ be a birth/death process with birth/death rates $\lambda_n$ and $\mu_n$ respectively. If \eqref{Limsup Birth/Death rates condition} holds then for each $k\in\N$ there exists $c_k>0$ such that
\begin{equation}
  \E_\infty\left[H_n^k\right]\leq c_k\frac{\lambda_n}{\mu_n},~~n\geq1.\label{H_n kth moment bound}
\end{equation}
\end{lemma}
\begin{proof}
We follow a similar method of proof as in \cite[Lemma 4.1]{bansaye2016speed} by first noting the random variable $H_n$ is simply the number of positive jumps before $T_{n-1}$ of a discrete-time random walk started at $n$ with transition probabilities $p_{i,i+1}=\lambda_i/(\lambda_i+\mu_i)$ and $p_{i,i-1}=\mu_i/(\lambda_i+\mu_i)$.

Next, we note that Lemma \ref{lemma3.1} also holds in the discrete-time case and so since \eqref{Limsup Birth/Death rates condition} holds we have that $H_n$, for $n\geq n_0$, is stochastically dominated by $T$, the hitting time of $n-1$ by a discrete-time simple random walk started at $n$ with up/down transition probabilities $p$, $1-p$ respectively; note that this is the same $p$ from \eqref{p def}. Since $p<1/2$, all moments of $T$---and thus $H_n$ for $n\geq n_0$---are finite. Finiteness of moments of $H_n$ for $n<n_0$ can be obtained through the following recursive relationship:
\[\E_\infty [H_n]=\frac{\lambda_n}{\lambda_n+\mu_n}(1+\E_\infty[H_{n+1}]+\E_\infty[H_n]).\]
We now consider the Laplace transform of $H_n$, denoted $\hat{G}_n(a):=\E_\infty\left[\text{exp}(-aH_n)\right]$, and the recursion formula (4.3) from \cite{bansaye2016speed}:
\begin{equation*}
  \hat{G}_n(a)=\frac{\mu_n}{\lambda_n+\mu_n}+\frac{\lambda_n}{\lambda_n+\mu_n}\text{exp}(-a)\hat{G}_n(a)\hat{G}_{n+1}(a), \qquad a\geq 0,\, n \geq 1.
\end{equation*}
Differentiating $k$ times gives us
\begin{equation*}
  \hat{G}_n^{(k)}(a)=\frac{\lambda_n}{\lambda_n+\mu_n}\sum_{\substack{|\boldsymbol{m}|=k\\\boldsymbol{m}\in\mathbb{N}_0^3}}\binom{k}{\boldsymbol{m}}(-1)^{m_1}\text{exp}(-a)\hat{G}_n^{(m_2)}(a)\hat{G}_{n+1}^{(m_3)}(a),
\end{equation*}
where $\N_0=\N\cup\{0\}$. If we now evaluate the above at $a=0$ we get
\begin{equation*}
  \E_\infty\left[H_n^k\right]=\frac{\lambda_n}{\lambda_n+\mu_n}\sum_{\substack{|\boldsymbol{m}|=k\\\boldsymbol{m}\in\mathbb{N}_0^3}}\binom{k}{\boldsymbol{m}}\E_\infty\left[H_n^{m_2}\right]\E_\infty\left[H_{n+1}^{m_3}\right],
\end{equation*}
since a factor of $(-1)^k$ can be cancelled from both sides. Rearranging this we see that
\[\frac{\mu_n}{\lambda_n}\E_\infty\left[H_n^k\right]=\sum_{\substack{|\boldsymbol{m}|=k\\\boldsymbol{m}\neq (0,k,0)}}\binom{k}{\boldsymbol{m}}\E_\infty\left[H_n^{m_2}\right]\E_\infty\left[H_{n+1}^{m_3}\right].\]
Now, the right hand side of this equation is a finite combination of moments; each of which are bounded by the respective moment of $T$. Thus the right hand side can be bounded by some constant $c_k\geq0$. This then gives us \eqref{H_n kth moment bound}.
\end{proof}
With this we now know enough about the asymptotic behaviour of the random variables in \eqref{Bound with Hn,n and Tn} to state the following:
\begin{proposition}\label{prop4.5}
  \begin{equation}
    \lim_{t\to0}\E\left[\sup_{s\leq t}\left(\frac{sN_s^{\sigma,\theta}}{2}-1\right)^k\right]=0,\qquad\forall k\in\mathbb{N}.\label{ASG E sup limit k}
  \end{equation}
\end{proposition}
\begin{proof}
  The first thing to note is that as soon as the right hand side of \eqref{From process to hitting times ineq} is bounded in expectation the result follows thanks to an application of \eqref{ASG Speed of CDI} and the dominated convergence theorem in a similar way to the proof of Proposition \ref{Kingman sup lim prop}.

  We start by looking at the terms on the right hand side of \eqref{Bound with Hn,n and Tn} for $j\neq0$. The aim here is to show that the expectations of these random variables are summable in $n$ and hence their supremum over $n$ has finite mean. Applying \eqref{H_n kth moment bound}, \eqref{ASG T_n Kingman comparison Ineq}, and the Cauchy-Schwarz inequality to these terms we obtain the following:
  \begin{align*}
    \E\left[\binom{k}{j}\frac{n^{k-j}}{2^k}\left(H_n^{\sigma,\theta}\right)^j\left(T_{n-1}^{\sigma,\theta}\right)^k\right]&\leq\binom{k}{j}\frac{n^{k-j}}{2^k}\E\left[\left(H_n^{\sigma,\theta}\right)^{2j}\right]^{1/2}\E\left[\left(T_{n-1}^{\sigma,\theta}\right)^{2k}\right]^{1/2}\\
    &\leq\binom{k}{j}\frac{n^{k-j}}{2^k}\frac{c_{2j}^{1/2}}{n^{1/2}}\left(\frac{C_{2k}^{0,\theta}}{n^{2k}}+\frac{C_{2k}^{\sigma,\theta}}{n^{2k+1}}\right)^{1/2}\\
    &=O\left(n^{-j-1/2}\right).
  \end{align*}
  This then gives us that the expectation of the terms on the right hand side of \eqref{Bound with Hn,n and Tn} are summable in $n$ as long as $j\neq0$. Thus we have that their supremum over $n$ is finite in expectation. 
  
  Next we need to deal with the first term on the right hand side of \eqref{Bound with Hn,n and Tn}; i.e.\ the term corresponding to $j=0$. Since
  \[\sup_{n\in\mathbb{N}}\frac{n^k}{2^k}\left(T_{n-1}^{\sigma,\theta}\right)^k\leq\sup_{n\in\mathbb{N}}\left[\frac{n^k}{2^k}\left(T_{n-1}^{\sigma,\theta}\right)^k-\frac{n^k}{2^k}\left(T_{n-1}^{0,\theta}\right)^k\right]+\sup_{n\in\mathbb{N}}\frac{n^k}{2^k}\left(T_{n-1}^{0,\theta}\right)^k\]
it suffices to show that the right hand side of the above is finite in expectation. For the second term this follows immediately from \eqref{Kingman sup bound k} since $n^k(T_{n-1}^{0,\theta})^k/2^k$ are the right end points of the continuous pieces of the process $t^k(N_t^{0,\theta})^k/2^k$. For the first term we consider the increments in $n$ of these random variables and show they are absolutely summable. This gives the desired result since, for any real sequence $x_n$,
\[\sup_{n\in\mathbb{N}}x_n\leq x_1+\sum_{n\geq1}|x_{n+1}-x_n|.\]
For us, the corresponding increment is
\begin{align}
  &\left|\frac{(n+1)^k}{2^k}\left(T_{n}^{\sigma,\theta}\right)^k-\frac{(n+1)^k}{2^k}\left(T_{n}^{0,\theta}\right)^k-\frac{n^k}{2^k}\left(T_{n-1}^{\sigma,\theta}\right)^k+\frac{n^k}{2^k}\left(T_{n-1}^{0,\theta}\right)^k\right|\nonumber\\
  &=\left|\left(\frac{(n+1)^k}{2^k}\left(T_{n}^{\sigma,\theta}\right)^k-\frac{n^k}{2^k}\left(T_{n-1}^{\sigma,\theta}\right)^k\right)-\left(\frac{(n+1)^k}{2^k}\left(T_{n}^{0,\theta}\right)^k-\frac{n^k}{2^k}\left(T_{n-1}^{0,\theta}\right)^k\right)\right|.\label{Increments Formula}
\end{align}
Now, the terms inside the brackets are just the difference between successive right endpoints of continuous pieces of the processes $\left(sN_s^{\sigma,\theta}/2\right)^k$ and $\left(sN_s^{0,\theta}/2\right)^k$ respectively. These can be expressed in the following way:
\begin{align}
  \lefteqn{\frac{(n+1)^k}{2^k}\left(T^{\sigma,\theta}_n\right)^k-\frac{n^k}{2^k}\left(T^{\sigma,\theta}_{n-1}\right)^k} \qquad \nonumber\\
  &= \left[(n+1)^k-n^k\right]\frac{\left(T^{\sigma,\theta}_n\right)^k}{2^k} -\frac{n^k}{2^k}\left(\left(T^{\sigma,\theta}_{n-1}\right)^k-\left(T^{\sigma,\theta}_n\right)^k\right)\nonumber\\
  &=\frac{\left(T^{\sigma,\theta}_n\right)^k}{2^k}\left(\sum_{j=1}^k\binom{k}{j}\frac{n^{k-j}}{2^k}\right) -\frac{n^k}{2^k}\left(\sum_{j=1}^k\binom{k}{j}\left(T^{\sigma,\theta}_n\right)^{k-j}\left(T^{\sigma,\theta}_{n,n-1}\right)^j\right), \label{Right endpoint difference}
\end{align}
since $(T^{\sigma,\theta}_{n-1})^k=(T^{\sigma,\theta}_n+T^{\sigma,\theta}_{n,n-1})^k$. Note that one can set $\sigma=0$ in the above meaning these equalities apply to both the Kingman coalescent and the ASG. Using the substitution \eqref{Right endpoint difference}, \eqref{Increments Formula} is then equal to
\begin{align}
  &\left|\sum_{j=1}^k\binom{k}{j}\frac{n^{k-j}}{2^k}\left[\left(T_n^{\sigma,\theta}\right)^k-\left(T_n^{0,\theta}\right)^k\right]\right.\nonumber\\
  &{}-\frac{n^k}{2^k}\left.\sum_{j=1}^k\binom{k}{j}\left(\left(T_n^{\sigma,\theta}\right)^{k-j}\left(T_{n,n-1}^{\sigma,\theta}\right)^j-\left(T_n^{0,\theta}\right)^{k-j}\left(T_{n,n-1}^{0,\theta}\right)^j\right)\right|\nonumber\\
  \leq{}&\sum_{j=1}^k\binom{k}{j}\frac{n^{k-j}}{2^k}\left|\left(T_n^{\sigma,\theta}\right)^k-\left(T_n^{0,\theta}\right)^k\right|\label{Increments Upper Bound}\\
  &{}+\frac{n^k}{2^k}\sum_{j=1}^k\binom{k}{j}\left|\left(T_n^{\sigma,\theta}\right)^{k-j}\left(T_{n,n-1}^{\sigma,\theta}\right)^j-\left(T_n^{0,\theta}\right)^{k-j}\left(T_{n,n-1}^{0,\theta}\right)^j\right|.\nonumber
\end{align}
If we now apply expectations to \eqref{Increments Upper Bound}, then the expectation of \eqref{Increments Formula} is bounded by
\begin{align*}
  \E&\left[\Big|\frac{(n+1)^k}{2^k}\left(T_{n}^{\sigma,\theta}\right)^k-\frac{(n+1)^k}{2^k}\left(T_{n}^{0,\theta}\right)^k-\frac{n^k}{2^k}\left(T_{n-1}^{\sigma,\theta}\right)^k+\frac{n^k}{2^k}\left(T_{n-1}^{0,\theta}\right)^k\Big|\right]\\
  \leq{}&\sum_{j=1}^k\binom{k}{j}\frac{n^{k-j}}{2^k}\E\left[\Big|\left(T_n^{\sigma,\theta}\right)^k-\left(T_n^{0,\theta}\right)^k\Big|\right]\\
  &{}+\frac{n^k}{2^k}\sum_{j=1}^k\binom{k}{j}\E\left[\Big|\left(T_n^{\sigma,\theta}\right)^{k-j}\left(T_{n,n-1}^{\sigma,\theta}\right)^j-\left(T_n^{0,\theta}\right)^{k-j}\left(T_{n,n-1}^{0,\theta}\right)^j\Big|\right]\\
  \leq{}&\sum_{j=1}^k\binom{k}{j}\frac{n^{k-j}}{2^k}\E\left[\Big|\left(T_n^{\sigma,\theta}\right)^k-\left(T_n^{0,\theta}\right)^k\Big|\right]\\
  &{}+\frac{n^k}{2^k}\sum_{j=1}^k\binom{k}{j}\E\left[\left(T_n^{\sigma,\theta}\right)^{k-j}\Big|\left(T_{n,n-1}^{\sigma,\theta}\right)^j-\left(T_{n,n-1}^{0,\theta}\right)^j\Big|\right]\\
  &{}+\frac{n^k}{2^k}\sum_{j=1}^k\binom{k}{j}\E\left[\left(T_{n,n-1}^{0,\theta}\right)^j\Big|\left(T_n^{\sigma,\theta}\right)^{k-j}-\left(T_n^{0,\theta}\right)^{k-j}\Big|\right].
\end{align*}
Recalling equations \eqref{PRM Ineq Kingman to Mutation} and \eqref{PRM Ineq Mutation to Selection} from Section \ref{Section2}, we see that on the joint probability space we have constructed the hitting times for the ASG are almost surely longer than the corresponding hitting times for the Kingman coalescent with mutation. This allows us to drop modulus signs from inside the above expectations. Combining this with \eqref{ASG T_n Kingman comparison Ineq}, \eqref{ASG Tn,n-1 kth moment ineqs}, and \eqref{Kingman T_n Asymptotics} results in the following inequality:
\begin{align*}
\lefteqn{\E\left[\Big|\frac{(n+1)^k}{2^k}\left(T_{n}^{\sigma,\theta}\right)^k-\frac{(n+1)^k}{2^k}\left(T_{n}^{0,\theta}\right)^k-\frac{n^k}{2^k}\left(T_{n-1}^{\sigma,\theta}\right)^k+\frac{n^k}{2^k}\left(T_{n-1}^{0,\theta}\right)^k\Big|\right]} \qquad\\
\leq {}& \sum_{j=1}^k\binom{k}{j}\frac{n^{k-j}}{2^k}\frac{C_k^{\sigma,\theta}}{n^{k+1}}\\
&{}+\frac{n^k}{2^k}\sum_{j=1}^k\binom{k}{j}\left[\left(\frac{C_{k-j}^{0,\theta}}{n^{k-j}}+\frac{C_{k-j}^{\sigma,\theta}}{n^{k-j+1}}\right)\frac{B_j^{\sigma,\theta}}{n^{2j+1}}+\frac{2^j}{n^j(n-1+\theta)^j}\frac{C_j^{\sigma,\theta}}{n^{k-j+1}}\right]\\
= {}& O(n^{-2}).
\end{align*}
From this we can conclude that the right hand side of \eqref{From process to hitting times ineq} is finite in expectation and so the result follows.
\end{proof}
\section{Proof of Theorem 1.1}\label{Section5}
In order to prove Theorem \ref{maintheorem} we extend the arguments of Limic and Talarczyk \cite{limic2015} with modifications where needed. Lemmas \ref{lemmaXdecomposition} and \ref{lemma5.3} are analogues of Lemmas 2.2 and 2.3 from \cite{limic2015} for the ASG. Though the methods of proof are similar, they rely heavily on our PRM construction of the ASG and the results from Section \ref{Section4}; most importantly Proposition \ref{prop4.5}. Regardless of the similarities we present the results here in full in order to keep the proof self-contained. 

Before establishing the key lemmas needed to prove Theorem \ref{maintheorem} we state a technical lemma from \cite{berestycki2010} that is used repeatedly in what follows.
\begin{lemma}[{\cite[Lemma 10]{berestycki2010}}]
  Suppose $f,g:[a,b]\rightarrow\mathbb{R}$ are c\`adl\`ag functions such that \[\sup_{x\in[a,b]}\limits\left|f(x)+\int_a^xg(u)\,{{\rm{d}}}u\right|\leq K,\]
  for some $K<\infty$. If, in addition $f(x)g(x)>0,~x\in[a,b]$, whenever $f(x)\neq 0$, then
  \begin{equation}
    \sup_{x\in[a,b]}\left|\int_a^xg(u)\,{{\rm{d}}}u\right|\leq K\text{ and}\sup_{x\in[a,b]}\left|f(x)\right|\leq2K.\label{TechnicalLemma}
  \end{equation}
\end{lemma}
The first lemma we establish gives a formula for $tN_t^{\sigma,\theta}/2$ which helps us in analysing its small time behaviour.
\begin{lemma} \label{lemmaXdecomposition}Under the assumptions of Theorem \ref{maintheorem} we have
  \begin{equation}
    \frac{tN_t^{\sigma,\theta}}{2}=1-\int_0^t\left(\frac{sN_s^{\sigma,\theta}}{2}-1\right)\frac{1}{s}\,{{\rm{d}}}s-M^{\sigma,\theta}_t+R^{\sigma,\theta}_t,~~~t\geq0,\label{New representation of tNt/2}
  \end{equation}
  where
  \begin{equation}
    M^{\sigma,\theta}_t=\frac{1}{2}\int_0^t\int_{\bar{\Delta}}s\mathbbm{1}_{\bar{\Delta}_{N_{s-}^{\sigma,\theta}}}(\boldsymbol{k})\left[\mathbbm{1}_{j>0}(\boldsymbol{k})-\mathbbm{1}_{j=0}(\boldsymbol{k})\right]\hat{\pi}({{\rm{d}}}s,{\rm{d}}\boldsymbol{k}),~~~t\geq0,\label{MartingaleFormula}
  \end{equation}
 and $R^{\sigma,\theta}$ is a continuous process such that for any $T>0$ there exists $C_1>0$ such that
  \begin{equation}
    \E\left[\sup_{s\leq t}|R^{\sigma,\theta}_s|\right]\leq C_1t,~~~t\leq T.\label{RProcessBound}
  \end{equation}
\end{lemma}
\begin{proof}
  Letting $0<r\leq t$, the PRM construction in Section \ref{Section2} allows us to write the process as follows
  \begin{align*}
    N_t^{\sigma,\theta}={}&N_r^{\sigma,\theta}-\int_{(r,t]}\int_{\bar{\Delta}}\mathbbm{1}_{\bar{\Delta}_{N_{s-}^{\sigma,\theta}}}(\boldsymbol{k})\left[\mathbbm{1}_{j>0}(\boldsymbol{k})-\mathbbm{1}_{j=0}(\boldsymbol{k})\right]\pi({{\rm{d}}}s,{\rm{d}}\boldsymbol{k}),\\
    ={}&N_r^{\sigma,\theta}-\int_{(r,t]}\int_{\bar{\Delta}}\mathbbm{1}_{\bar{\Delta}_{N_{s-}^{\sigma,\theta}}}(\boldsymbol{k})\left[\mathbbm{1}_{j>0}(\boldsymbol{k})-\mathbbm{1}_{j=0}(\boldsymbol{k})\right]\nu({\rm{d}}s,{\rm{d}}\boldsymbol{k})\\
    &{}-\int_{(r,t]}\int_{\bar{\Delta}}\mathbbm{1}_{\bar{\Delta}_{N_{s-}^{\sigma,\theta}}}(\boldsymbol{k})\left[\mathbbm{1}_{j>0}(\boldsymbol{k})-\mathbbm{1}_{j=0}(\boldsymbol{k})\right]\hat{\pi}({\rm{d}}s,{\rm{d}}\boldsymbol{k}),\\
    ={}&N_r^{\sigma,\theta}-\int_{(r,t]}\frac{N_s^{\sigma,\theta}(N_s^{\sigma,\theta}-1+\theta-\sigma)}{2}\,{\rm{d}}s\\
    &{}-\int_{(r,t]}\int_{\bar{\Delta}}\mathbbm{1}_{\bar{\Delta}_{N_{s-}^{\sigma,\theta}}}(\boldsymbol{k})\left[\mathbbm{1}_{j>0}(\boldsymbol{k})-\mathbbm{1}_{j=0}(\boldsymbol{k})\right]\hat{\pi}({\rm{d}}s,{\rm{d}}\boldsymbol{k}).
  \end{align*}
 Since all jump times are isolated and countable then this representation is permissible. From integration by parts we also get that
  \[tN_t^{\sigma,\theta}=rN_r^{\sigma,\theta}+\int_{(r,t]}N_s^{\sigma,\theta}\,{\rm{d}}s+\int_{(r,t]}s\,{\rm{d}}N_s^{\sigma,\theta}.\]
  Combining these we obtain
  \begin{align}
    \frac{tN_t^{\sigma,\theta}}{2}=&\frac{rN_r^{\sigma,\theta}}{2}+\int_{(r,t]}\left(\frac{N_s^{\sigma,\theta}}{2}-s\frac{N_s^{\sigma,\theta}(N_s^{\sigma,\theta}-1+\theta-\sigma)}{4}\right)\,{\rm{d}}s\nonumber\\
    &\hspace{2cm}-\int_{(r,t]}\int_{\bar{\Delta}}\frac{s}{2}\mathbbm{1}_{\bar{\Delta}_{N_{s-}^{\sigma,\theta}}}(\boldsymbol{k})\left[\mathbbm{1}_{j>0}(\boldsymbol{k})-\mathbbm{1}_{j=0}(\boldsymbol{k})\right]\hat{\pi}({\rm{d}}s,{\rm{d}}\boldsymbol{k}).\label{Partial tNt representation, with r}
  \end{align}
  Next we need to check that one can formally let $r=0$ in the above expression, recognise that the final term is then equal to $M^{\sigma,\theta}_t$, and rearrange the drift term to reflect the formula in \eqref{New representation of tNt/2}.
  
  Starting with the final term in \eqref{Partial tNt representation, with r}, we first need to fix $T\geq0$ and use \eqref{ASG E sup limit k} to find that, for $t\leq T$
  \begin{equation}
      \E\left[\sup_{s\leq t}\left(\frac{sN_s^{\sigma,\theta}}{2}\right)^k\right]<\infty.\label{ASG E sup k bound}
  \end{equation}
  We can then use properties of the compensated Poisson integral to bound the second moment of the martingale \eqref{MartingaleFormula}:
  \begin{align}
    \E[(M^{\sigma,\theta}_t)^2]&=\E\left[\int_0^t\int_{\bar{\Delta}}\frac{s^2}{4}\mathbbm{1}_{\bar{\Delta}_{N^{\sigma,\theta}_{s-}}}(\boldsymbol{k})\nu({\rm{d}}s,{\rm{d}}\boldsymbol{k})\right]\nonumber\\
    &=\E\left[\int_0^ts^2\frac{N_s^{\sigma,\theta}(N_s^{\sigma,\theta}-1+\theta+\sigma)}{8}\,{\rm{d}}s\right]\nonumber\\
    &\leq\E\left[\int_0^t s^2\frac{(N_s^{\sigma,\theta})^2}{8}+Ts\frac{(\theta+\sigma)N_s^{\sigma,\theta}}{8}\,{\rm{d}}s\right]\label{MgL2Bound part1}\\
    &\leq C_2t,\qquad t\leq T, \label{MgL2Bound}
  \end{align}
  where $C_2\geq0$ comes from \eqref{ASG E sup k bound} with $k=1,2$.
  
  It should be noted that the constant above (and most of the constants that follow) have an implicit dependence on $T$; though since we are concerned with behaviour near zero this dependence is unimportant.
  Thanks to \cite[Theorem 8.23]{peszat_zabczyk_2007} and \eqref{MgL2Bound} we get that $M^{\sigma,\theta}$ given by \eqref{MartingaleFormula} is a well defined square integrable martingale. Moreover, Doob's $L^2$ maximal inequality gives us the following bound:
  \begin{equation}
    \E\left[\sup_{s\leq t}(M^{\sigma,\theta}_s)^2\right]\leq 4C_2t,\qquad\forall t\leq T.\label{MgSupL2Bound}
  \end{equation}
  We note that the final term in \eqref{Partial tNt representation, with r} is equal to $M^{\sigma,\theta}_t-M^{\sigma,\theta}_r$ and by \eqref{MgSupL2Bound} we have $M^{\sigma,\theta}_r\to 0$ in $L^2$ as $r\to0$.
  
  Next, we rewrite the integral with respect to $s$ in \eqref{Partial tNt representation, with r} as
  \begin{align}
    A^{\sigma,\theta}_r(t)&:=\frac{1}{2}\int_{(r,t]}N^{\sigma,\theta}_s\,{\rm{d}}s-\frac{1}{2}\int_{(r,t]}s\frac{N_s^{\sigma,\theta}(N_s^{\sigma,\theta}-1+\theta-\sigma)}{2}\,{\rm{d}}s\nonumber\\
    &\phantom{:}=-\int_{(r,t]}\frac{N_s^{\sigma,\theta}}{2}\left(\frac{s}{2}N_s^{\sigma,\theta}-1\right)\,{\rm{d}}s+\int_{(r,t]}\frac{s(1-\theta+\sigma)}{4}N_s^{\sigma,\theta}\,{\rm{d}}s.\label{ArDef}
  \end{align}
  This allows us to rearrange \eqref{Partial tNt representation, with r} into the following:
  \begin{align*}
    \frac{t}{2}N_t^{\sigma,\theta}-1+\int_{(r,t]}\frac{N_s^{\sigma,\theta}}{2}\left(\frac{s}{2}N_s^{\sigma,\theta}-1\right)\,{\rm{d}}s={}&\frac{r}{2}N_r^{\sigma,\theta}-1+\int_{(r,t]}\frac{(1-\theta+\sigma)s}{4}N_s^{\sigma,\theta}\,{\rm{d}}s\\
    &{}-(M^{\sigma,\theta}_t-M^{\sigma,\theta}_r).
  \end{align*}
  Applying \eqref{TechnicalLemma} with $f(s)=s N_s^{\sigma,\theta}/2-1$, $g(s)=N_s^{\sigma,\theta}(s N_s^{\sigma,\theta}/2-1)/2$, $a=r$, and $b=t$ we find that
  \begin{align*}
    \sup_{r\leq s\leq t}\left|\frac{s}{2}N_s^{\sigma,\theta}-1\right|\leq{}& 2\Bigg(\left|\frac{r}{2}N_r^{\sigma,\theta}-1\right|+|M^{\sigma,\theta}_r|\\
    {}&+\sup_{r\leq s\leq t}|M^{\sigma,\theta}_s|+\int_{(r,t]}\left|\frac{(1-\theta+\sigma)s}{4}N_s^\theta\right|\,{\rm{d}}s\Bigg).
  \end{align*}
  Letting $r\to0$ in the above and using \eqref{ASG Speed of CDI} gives us
  \begin{equation}
    \sup_{s\leq t}\left|\frac{s}{2}N_s^{\sigma,\theta}-1\right|\leq 2\left(\sup_{s\leq t}|M^{\sigma,\theta}_s|+\int_0^t\left|\frac{(1-\theta+\sigma)s}{4}N_s^{\sigma,\theta}\right|\,{\rm{d}}s\right),\label{Improving sup convergence}
  \end{equation}
  where letting $r=0$ in the integral on the right hand side is permissible thanks to \eqref{ASG E sup k bound} with $k=1$. Squaring both sides of \eqref{Improving sup convergence}, applying expectations and using \eqref{MgSupL2Bound} and \eqref{ASG E sup k bound} gives us that there exists $C_3>0$ such that
  \begin{equation}
    \E\left[\sup_{s\leq t}\left(\frac{sN_s^{\sigma,\theta}}{2}-1\right)^2\right]\leq C_3t,\qquad\forall t\leq T.\label{New ASG Sup squared linear bound}
  \end{equation}
  We can now use this bound, which improves on \eqref{ASG E sup limit k}, to ensure the integral term in \eqref{New representation of tNt/2} is well-defined. Letting $X^{\sigma,\theta}_t:=X_1^{\sigma,\theta}(t)=tN_t^{\sigma,\theta}/2-1$, \eqref{New ASG Sup squared linear bound} (along with Jensen's inequality) allows us to control $X_t^{\sigma,\theta}$ as follows:
  \begin{equation}
    -\sqrt{C_3t}\leq\E\left[X_t^{\sigma,\theta}\right]\leq\sqrt{C_3t},\label{SqrtProcessBound}
  \end{equation}
  for $t\leq T$. If we now consider the integral term in \eqref{New representation of tNt/2} we see that by \eqref{SqrtProcessBound}
  \begin{align*}
    \E\left[\int_0^tX^{\sigma,\theta}_s\frac{1}{s}\,{\rm{d}}s\right]=\int_0^t\E\left[X^{\sigma,\theta}_s\right]\frac{1}{s}\,{\rm{d}}s
    \leq\int_0^t \frac{\sqrt{C_3}}{\sqrt{s}}\,{\rm{d}}s
    =2\sqrt{C_3t}.
  \end{align*}
  This shows that the integral with respect to $s$ in \eqref{New representation of tNt/2} has finite expectation and thus is finite a.s.~and so is well-defined.
  
  Next, we can express $A^{\sigma,\theta}_r(t)$ as
  \begin{align*}
      A^{\sigma,\theta}_r(t)={}&-\int_{(r,t]}\left(\frac{s}{2}N_s^{\sigma,\theta}-1\right)^2\frac{1}{s}\,{\rm{d}}s-\int_{(r,t]}\left(\frac{s}{2}N_s^{\sigma,\theta}-1\right)\frac{1}{s}\,{\rm{d}}s\\
      &{}+\frac{1}{2}\int_{(r,t]}\frac{(1-\theta+\sigma)s}{2}N_s^{\sigma,\theta}\,{\rm{d}}s.
  \end{align*}
    Using \eqref{New ASG Sup squared linear bound} along with \eqref{ASG E sup k bound} we find that
    \[\E\left[\Big|A^{\sigma,\theta}_r(t)+\int_{(r,t]}\left(\frac{s}{2}N_s^{\sigma,\theta}-1\right)\frac{1}{s}\,{\rm{d}}s\Big|\right]\leq C_4t,\qquad\forall t\leq T,\]
    where $C_4>0$ does not depend on $r$. This shows that, as $r\to0$, $A^{\sigma,\theta}_r(t)$ converges in $L^1$ to
    \[-\int_0^t\left(\frac{sN_s^{\sigma,\theta}}{2}-1\right)\frac{1}{s}\,{\rm{d}}s+R^{\sigma,\theta}_t,\]
    where
   \begin{equation}
     R^{\sigma,\theta}_t:=-\int_0^t\left(\frac{s}{2}N_s^{\sigma,\theta}-1\right)^2\frac{1}{s}\,{\rm{d}}s+\frac{1}{2}\int_0^t\frac{(1-\theta+\sigma)s}{2}N_s^{\sigma,\theta}\,{\rm{d}}s.\label{RtDef}
   \end{equation}
  To conclude, \eqref{New ASG Sup squared linear bound} applied to the first term on the right hand side of \eqref{RtDef} and \eqref{ASG E sup k bound} applied to the second term yield the bound in \eqref{RProcessBound}.
\end{proof}
Next, we will reduce the problem of convergence of \eqref{ASG 2nd order pre limit} to convergence of the following process defined via the martingale \eqref{MartingaleFormula}:
\begin{align}
    Y_t^{\sigma,\theta}&:=-\frac{1}{t}\int_0^tu\,{\rm{d}}M_u^{\sigma,\theta},~~t>0,~~Y_0^{\sigma,\theta}=0,\label{Y def}\\
    Y_\epsilon^{\sigma,\theta}(t)&:=\epsilon^{-1/2}Y_{\epsilon t}^{\sigma,\theta}.\label{Y epsilon def}
\end{align}
\begin{lemma}\label{lemma5.3}
The process $(Y_t^{\sigma,\theta})_{t\geq0}$ satisfies the equation
\begin{equation}
    Y_t^{\sigma,\theta}=-\int_0^tY_s^{\sigma,\theta}\frac{1}{s}\,{\rm{d}}s-M_t^{\sigma,\theta}.\label{Y integral representation}
\end{equation}
Moreover, there exists $C_5>0$ such that for any $t\leq T$
\begin{equation}
    \E\left[\sup_{s\leq t}\left(Y_s^{\sigma,\theta}\right)^2\right]\leq C_5t.\label{Y sup squared control}
\end{equation}
Finally, we have that
\begin{equation}
    \lim_{\epsilon\to0}\E\left[\sup_{s\leq t}\left|X_\epsilon^{\sigma,\theta}(s)-Y_\epsilon^{\sigma,\theta}(s)\right|\right]=0.\label{sup X-Y limit}
\end{equation}
\end{lemma}
\begin{proof}
  First we consider $L^{\sigma,\theta}_t:=tY^{\sigma,\theta}_t$ so that $L^{\sigma,\theta}$ is a square integrable martingale with quadratic variation
  \[\left[L^{\sigma,\theta}\right]_t=\frac{1}{4}\int_0^t\int_{\bar{\Delta}} s^4\mathbbm{1}_{\bar{\Delta}_{N_{s-}^{\sigma,\theta}}} (\boldsymbol{k})\pi({\rm{d}}s,{\rm{d}}\boldsymbol{k}),\]
  which has expectation
  \begin{align}
    \E \left[[L^{\sigma,\theta}]_t\right] &=\frac{1}{4}\E\left[\int_0^t\int_{\bar{\Delta}}s^4\mathbbm{1}_{\bar{\Delta}_{N_{s-}^{\sigma,\theta}}} (\boldsymbol{k})\nu({\rm{d}}s,{\rm{d}}\boldsymbol{k})\right]\nonumber\\
    &=\frac{1}{4}\E\left[\int_0^ts^2\left(s^2\frac{N_s^{\sigma,\theta}(N_s^{\sigma,\theta}-1+\theta+\sigma)}{2}\right)\,{\rm{d}}s\right]\label{Ltcalc2ndline}\\
    &\leq C_2\int_0^ts^2\,{\rm{d}}s=\frac{C_2}{3}t^3,\label{LtQuadvarIneq}
  \end{align}
  where we use the same bound from \eqref{MgL2Bound} on the expectation of the bracketed term in \eqref{Ltcalc2ndline}.
  
  This now gives us that
  \begin{equation}
    \E \left[\left(Y^{\sigma,\theta}_t\right)^2\right]=\frac{1}{t^2}\E \left[\left(L^{\sigma,\theta}_t\right)^2\right]=\frac{1}{t^2}\E \left[[L^{\sigma,\theta}]_t\right]\leq\frac{C_2}{3}t.\label{Y second moment bound}
  \end{equation}
  In order to obtain \eqref{Y integral representation} we first use integration by parts to write
  \begin{align}
    Y^{\sigma,\theta}_t-Y^{\sigma,\theta}_r&=\int_{(r,t]}\frac{1}{s^2}\int_0^s u\,{\rm{d}}M^{\sigma,\theta}_u\,{\rm{d}}s-\int_{(r,t]} \,{\rm{d}}M^{\sigma,\theta}_s \nonumber\\
    &=-\int_{(r,t]}\frac{1}{s}Y^{\sigma,\theta}_s\,{\rm{d}}s-M^{\sigma,\theta}_t+M^{\sigma,\theta}_r,\label{2.18}
  \end{align}
  and let $r\to0$; using \eqref{Y second moment bound} and Jensen's inequality to bound $\E[\int_{(r,t]}|Y^{\sigma,\theta}_s|/s\,{\rm{d}}s]$ uniformly in $r>0$ by $2\sqrt{Ct}/\sqrt{3}$, ensuring that $\int_0^tY^{\sigma,\theta}_s/s\,{\rm{d}}s$ exists almost surely.
  
  Now, we apply \eqref{TechnicalLemma} with $f(s):=Y^{\sigma,\theta}_s$, $g(s):=Y^{\sigma,\theta}_s/s$, $a=0$ and $b=t$ to get that
  \[\sup_{s\leq t}|Y^{\sigma,\theta}_s|\leq 2\sup_{s\leq t}M^{\sigma,\theta}_s.\]
  We can then square both sides and apply expectations along with \eqref{MgSupL2Bound} to obtain \eqref{Y sup squared control}.
  
  To prove \eqref{sup X-Y limit} we first recall that $X^{\sigma,\theta}_t=t N_t^\theta/2-1$ and so by Lemma \ref{lemmaXdecomposition} and \eqref{Y integral representation},
  \begin{equation}
    X^{\sigma,\theta}_t-Y^{\sigma,\theta}_t=-\int_0^t(X^{\sigma,\theta}_s-Y^{\sigma,\theta}_s)\frac{1}{s}\,{\rm{d}}s+R^\theta_t.\label{XminusYFormula}
  \end{equation}
  Applying \eqref{TechnicalLemma} to the above we find that
  \[\sup_{s\leq t}|X^{\sigma,\theta}_s-Y^{\sigma,\theta}_s|\leq 2\sup_{s\leq t}|R^{\sigma,\theta}_s|.\]
  This leads to
  \begin{align*}
    \E\left[\sup_{s\leq t}\left|X^{\sigma,\theta}_\epsilon(s)-Y^{\sigma,\theta}_\epsilon(s)\right|\right]&=\epsilon^{-1/2}\E\left[\sup_{s\leq t}\left|X^{\sigma,\theta}(\epsilon s)-Y^{\sigma,\theta}(\epsilon s)\right|\right]\\
    &\leq 2\epsilon^{-1/2}\E\left[\sup_{s\leq t}|R^{\sigma,\theta}_{\epsilon s}|\right]\\
    &\leq 2C_1\sqrt{\epsilon}t,
  \end{align*}
  using \eqref{RProcessBound}. Letting $\epsilon\to0$ in the above gives us \eqref{sup X-Y limit}.
\end{proof}
We now have everything in place to prove Theorem \ref{maintheorem}. The proof consists of first checking the conditions of \cite[Chapter 7, Theorem 1.4]{EK} are satisfied for the process
\[L^{\sigma,\theta}_\epsilon(t):=-tY^{\sigma,\theta}_\epsilon(t)=\frac{1}{\epsilon^{3/2}}\int_0^{\epsilon t}u\,{\rm{d}}M^{\sigma,\theta}_u,\qquad t\geq 0,\]
and proving its convergence to $(tZ_t)_{t\geq0}$ as $\epsilon\to0$. Then, as in \cite{limic2015}, we take advantage of Steps 2--4 in the proof of \cite[Lemma 4.8]{limic2015two} to extend this to the convergence of $Y_\epsilon^{\sigma,\theta}$ to $Z$ as $\epsilon\to0$. Essentially, the continuity of $t\mapsto 1/t$ away from zero along with the control we have over these processes near zero is what ensures this convergence. Finally, \eqref{sup X-Y limit} then gives us the convergence of \eqref{ASG 2nd order pre limit} to \eqref{Gaussian Process Limit for ASG + Kingman}.
\begin{proof}[Proof of Theorem \ref{maintheorem}]
  Starting then with $L_\epsilon^{\sigma,\theta}$, we first note that since $L^{\sigma,\theta}_\epsilon(t)=-\epsilon^{-3/2}L^{\sigma,\theta}(\epsilon t)$, then $L^{\sigma,\theta}_\epsilon$ is an $L^2$-martingale of the form
  \begin{equation}
    L^{\sigma,\theta}_\epsilon(t)=\epsilon^{-3/2}\frac{1}{2}\int_0^{\epsilon t}\int_{\bar{\Delta}}s^2\mathbbm{1}_{\bar{\Delta}_{N_{s-}^{\sigma,\theta}}}(\boldsymbol{k})\left[\mathbbm{1}_{j>0}(\boldsymbol{k})-\mathbbm{1}_{j=0}(\boldsymbol{k})\right]\hat{\pi}({\rm{d}}s,{\rm{d}}\boldsymbol{k}).\label{LepsilonFormula}
  \end{equation}
  The compensator of the square of this process is given by
  \begin{align}
    \langle L^{\sigma,\theta}_\epsilon\rangle_t&=\frac{1}{4\epsilon^3}\int_0^{\epsilon t}\int_{\bar{\Delta}}s^4\mathbbm{1}_{\bar{\Delta}_{N_{s-}^{\sigma,\theta}}}(\boldsymbol{k})\nu({\rm{d}}s,{\rm{d}}\boldsymbol{k}),\nonumber\\
    &=\frac{1}{4\epsilon^3}\int_0^{\epsilon t}s^4\frac{N_s^{\sigma,\theta}(N_s^{\sigma,\theta}-1+\theta+\sigma)}{2}\,{\rm{d}}s\nonumber\\
    &=\frac{1}{2}\int_0^ts^2(\epsilon s)^2\frac{N_{\epsilon s}^{\sigma,\theta}(N_{\epsilon s}^{\sigma,\theta}-1+\theta+\sigma)}{4}\,{\rm{d}}s.\label{LbracketEqn}
  \end{align}
   We now wish to verify that the assumptions (b) in \cite[Theorem 1.4, Chapter 7]{EK} are satisfied with $M_n$ corresponding to $L^{\sigma,\theta}_\epsilon$, $A_n$ corresponding to $\langle L^{\sigma,\theta}_\epsilon\rangle$ and $C_t=\int_0^tu^2\,{\rm{d}}u/2$.
  
  First, we note that since $\langle L^{\sigma,\theta}_\epsilon\rangle$ is continuous we only need to show (i) that $\langle L^{\sigma,\theta}_\epsilon\rangle_t$ converges as $\epsilon\to 0$ to $\int_0^tu^2/2\,{\rm{d}}u$ in probability for each $t>0$, and (ii) that for any $T>0$,
  \begin{equation}
    \lim_{\epsilon\to0}\E\left[\sup_{t\leq T}|L^{\sigma,\theta}_\epsilon(t)-L^{\sigma,\theta}_\epsilon(t-)|^2\right]=0.\label{Lepsilonjumpcontrol}
  \end{equation}
  For the first claim we consider the expected value of \eqref{LbracketEqn} as $\epsilon\to0$. Since the integrand is non-negative we can use Tonelli's theorem to switch the order of integration:
  \begin{align*}
    &\lim_{\epsilon\to0}\E\left[\frac{1}{2}\int_0^ts^2(\epsilon s)^2\frac{N_{\epsilon s}^{\sigma,\theta}(N_{\epsilon s}^{\sigma,\theta}-1+\theta+\sigma)}{4}\,{\rm{d}}s\right]\\
    &{}=\lim_{\epsilon\to0}\frac{1}{2}\int_0^ts^2\E\left[(\epsilon s)^2\frac{N_{\epsilon s}^{\sigma,\theta}(N_{\epsilon s}^{\sigma,\theta}-1+\theta+\sigma)}{4}\right]\,{\rm{d}}s.
  \end{align*}
  Similarly to \eqref{MgL2Bound}, we can control $\E\left[(\epsilon s)^2N_{\epsilon s}^{\sigma,\theta}(N_{\epsilon s}^{\sigma,\theta}-1+\theta+\sigma)/4\right]$ uniformly in $\epsilon>0$, allowing us to bring the limit inside the integral. Expanding the brackets we now get
  \[\frac{(\epsilon s)^2(N_{\epsilon s}^{\sigma,\theta})^2}{4}+\frac{(\sigma+\theta-1)(\epsilon s)^2N_{\epsilon s}^{\sigma,\theta}}{4}.\]
  The first term converges to 1 and the second term goes to 0 both in $L^1$ thanks to \eqref{ASG E sup limit k}. Thus (\ref{LbracketEqn}) converges to $\int_0^tu^2\,{\rm{d}}u/2$ in probability as $\epsilon\to0$.
  
  To address (ii): the limit in \eqref{Lepsilonjumpcontrol} holds since, thanks to the representation (\ref{LepsilonFormula}), the jumps of $L^{\sigma,\theta}_\epsilon$ on $[0,T]$ are isolated and uniformly bounded by $\epsilon^{-3/2}(\epsilon T)^2/2$ which goes to zero as $\epsilon\to0$. Thus we have the convergence of $L^{\sigma,\theta}_\epsilon\to (tZ_t)_{t\geq0}$ in law in $D_\mathbb{R}[0,\infty)$.
  
  Before concluding the proof one needs the following bound on the limiting process $Z$:
  \[\E\left[\sup_{s\leq t}Z_s^2\right]\leq C_6t,\qquad C_6\geq0.\]
  This bound follows from the formula
  \[Z_t=-\int_0^tZ_s\frac{1}{s}\,{\rm{d}}s+\frac{1}{\sqrt{2}}W_t,\]
  in combination with \eqref{TechnicalLemma}, similarly to the proof of \eqref{Y sup squared control}.
  
  The rest of the proof uses Steps 2--4 of \cite[Lemma 4.8]{limic2015two} to show that convergence of $L^{\sigma,\theta}_\epsilon$ implies the required convergence of $X_\epsilon^{\sigma,\theta}$; see also the proof of Theorem \ref{maintheorem} of \cite{limic2015}.
\end{proof}
\section*{Acknowledgements}
This work was supported by the EPSRC as well as the MASDOC DTC under grant EP/HO23364/1, by the Alan Turing Institute under the EPSRC grant EP/N510129/1, and by the EPSRC under grant EP/R044732/1.

%Research of PJ is supported by The Alan Turing Institute under the EPSRC grant EP/N510129/1.
%JK was supported by EPSRC grant EP/R044732/1.
%Commented this out as I didn't want to delete it, my suggestion as to what to write here is basically the same as what Jaro had in his paper. Paul I don't know if you want to state separately that it specifically was you who was supported by the Alan Turing Institute, I leave this to your discretion.

\bibliography{Bib}

\begin{thebibliography}{10}

\bibitem{KINGMAN1982235}
J.~Kingman, ``The coalescent,'' {\em Stochastic Processes and their
  Applications}, vol.~13, no.~3, pp.~235 -- 248, 1982.

\bibitem{berestycki2009recent}
N.~Berestycki, ``Recent progress in coalescent theory,'' {\em Ensaios
  Matem\'aticos}, vol.~16, pp.~1--193, 2009.

\bibitem{Griffiths1984AsymptoticLD}
R.~C. Griffiths, ``Asymptotic line-of-descent distributions,'' {\em Journal of
  Mathematical Biology}, vol.~21, pp.~67--75, 1984.

\bibitem{Jenkins_2017}
P.~A. Jenkins and D.~Span\`o, ``Exact simulation of the {Wright}--{Fisher}
  diffusion,'' {\em The Annals of Applied Probability}, vol.~27,
  p.~1478–1509, Jun 2017.

\bibitem{bansaye2016speed}
V.~Bansaye, S.~M\'el\'eard, and M.~Richard, ``Speed of coming down from
  infinity for birth-and-death processes,'' {\em Advances in Applied
  Probability}, vol.~48, no.~4, p.~1183–1210, 2016.

\bibitem{limic2015}
V.~Limic and A.~Talarczyk, ``Diffusion limits at small times for
  {$\Lambda$}-coalescents with a {Kingman} component,'' {\em Electron. J.
  Probab.}, vol.~20, p.~20 pp., 2015.

\bibitem{Neuhauser519}
C.~Neuhauser and S.~M. Krone, ``The genealogy of samples in models with
  selection,'' {\em Genetics}, vol.~145, no.~2, pp.~519--534, 1997.

\bibitem{KRONE1997210}
S.~M. Krone and C.~Neuhauser, ``Ancestral processes with selection,'' {\em
  Theoretical Population Biology}, vol.~51, no.~3, pp.~210 -- 237, 1997.

\bibitem{GriffithsMarjoramRecomb97}
R.~Griffiths and P.~Marjoram, ``An ancestral recombination graph,'' in {\em
  Progress in population genetics and human evolution}, pp.~257 -- 270,
  Springer, 1997.
\newblock Progress in population genetics and human evolution ; Conference
  date: 01-01-1997.

\bibitem{donnelly1999}
P.~Donnelly and T.~G. Kurtz, ``Genealogical processes for fleming-viot models
  with selection and recombination,'' {\em Ann. Appl. Probab.}, vol.~9,
  pp.~1091--1148, 11 1999.

\bibitem{berestycki2010}
J.~Berestycki, N.~Berestycki, and V.~Limic, ``The {$\Lambda$}-coalescent speed
  of coming down from infinity,'' {\em Ann. Probab.}, vol.~38, pp.~207--233, 01
  2010.

\bibitem{peszat_zabczyk_2007}
S.~Peszat and J.~Zabczyk, {\em Stochastic Partial Differential Equations with
  L\'evy Noise: An Evolution Equation Approach}.
\newblock Encyclopedia of Mathematics and its Applications, Cambridge
  University Press, 2007.

\bibitem{EK}
S.~N. Ethier and T.~G. Kurtz, {\em Markov processes -- characterization and
  convergence}.
\newblock Wiley Series in Probability and Mathematical Statistics: Probability
  and Mathematical Statistics, New York: John Wiley \& Sons Inc., 1986.

\bibitem{limic2015two}
V.~Limic and A.~Talarczyk, ``Second-order asymptotics for the block counting
  process in a class of regularly varying ${\Lambda}$-coalescents,'' {\em Ann.
  Probab.}, vol.~43, pp.~1419--1455, 05 2015.

\end{thebibliography}
\bibliographystyle{ieeetr}

\end{document}